\documentclass[11pt]{amsart}
\usepackage[T1]{fontenc}
\usepackage{amssymb,amscd,amsthm,latexsym}
\allowdisplaybreaks[3]
\usepackage{graphicx}
\usepackage[british]{babel}
\usepackage[all]{xy}
\usepackage{color}
\usepackage{mathabx}
\usepackage{calrsfs}
\usepackage{booktabs}

\theoremstyle{plain}
\newtheorem{theorem}[subsection]{Theorem}
\newtheorem{lemma}[subsection]{Lemma}
\newtheorem{proposition}[subsection]{Proposition}
\newtheorem{corollary}[subsection]{Corollary}

\theoremstyle{definition}

\theoremstyle{remark}

\newtheorem{remark}[subsection]{Remark}

\newenvironment{tfae}{
\begin{enumerate}}{\end{enumerate}}

\newcommand{\A}{\ensuremath{\mathbb{A}}}
\newcommand{\B}{\ensuremath{\mathbb{B}}}
\newcommand{\Cc}{\ensuremath{\mathsf{C}}}

\newcommand{\Ord}{\ensuremath{\mathsf{Ord}}}
\newcommand{\Rel}{\ensuremath{\mathsf{Rel}}}
\newcommand{\Gph}{\ensuremath{\mathsf{Gph}}}

\def\mathrmdef#1{\expandafter\def\csname#1\endcsname{{\rm#1}}}
\mathrmdef{co}\mathrmdef{lan}\mathrmdef{ran}\mathrmdef{ev}\mathrmdef{op}

\def\pullback{
 \ar@{-}[]+R+<6pt,-1pt>;[]+RD+<6pt,-6pt>%
 \ar@{-}[]+D+<1pt,-6pt>;[]+RD+<6pt,-6pt>}
\def\ophalfsplitpullback{%
 \ar@{-}[]+R+<6pt,-1pt>;[]+RD+<6pt,-6pt>%
 \ar@{-}[]+D+<.5ex,-6pt>;[]+RD+<6pt,-6pt>}

\usepackage{chngcntr}
\counterwithin*{equation}{section}

\begin{document}

\title{Lax comma categories of ordered sets}
\author{Maria Manuel Clementino}
\address{University of Coimbra, CMUC, Department of Mathematics, 3000-143 Coimbra, Portugal}\thanks{}
\email{mmc@mat.uc.pt}

\author{Fernando Lucatelli Nunes}
\address{Departement Informatica, Universiteit Utrecht, Nederland \& University of Coimbra,
CMUC, Department of Mathematics, Portugal}
\thanks{The research was supported through the programme “Oberwolfach Leibniz Fellows” by the Mathematisches Forschungsinstitut Oberwolfach in 2022, and partially supported by the Centre for Mathematics of the University of Coimbra - UIDB/00324/2020, funded by the Portuguese Government through FCT/MCTES}
\email{f.lucatellinunes@uu.nl}

\keywords{effective descent morphisms, lax comma $2$-categories, comma categories,  exponentiability, cartesian closed categories, topological functors, enriched categories, $\Ord$-enriched categories}

\subjclass{06A07, 18A25, 18A30, 18N10, 18D20, 18E50}

\begin{abstract}
Let $\Ord $ be the category of (pre)ordered sets.
Unlike $\Ord/X$, whose behaviour is well-known, not much can be found in the literature about the lax comma 2-category $\Ord//X$. In this paper we show that  the forgetful functor $\Ord//X\to \Ord$ is topological if and only if $X$ is complete. Moreover, under suitable hypothesis, $\Ord // X$  is complete and cartesian closed if and only if $X$ is. We end by analysing descent in this category. Namely, when $X$ is complete, we show that, for a morphism in $\Ord//X$, being pointwise effective for descent in $\Ord$ is sufficient, while being effective for descent in $\Ord$ is necessary, to be effective for descent in $\Ord//X$.
\end{abstract}

\maketitle

\section*{Introduction}
Janelidze's Galois theory~\cite{MR1061480, MR1822890} neatly gives a common ground for many Galois-type theories, prominently including Magid's Galois theory of commutative rings, Grothendieck's theory of \'{e}tale covering of schemes, and central extension of groups.
There is a deep connection between Janelidze's Galois theory and factorization systems~\cite{JanK, CJK}.

Motivated by this connection and the theory of \textit{lax orthogonal factorization
	systems}~\cite{MR3545937, MR3708821}, we have started a project  whose aim is to
investigate
two-dimensional extensions of the basic
ideas and results of Janelidze's Galois theory.

It has been noticeable that the so-called \textit{lax comma $2$-categories} play an important role in our work (\textit{c.f.} \cite{CLN}). Although they are quite natural (appearing, for instance, in \cite{ML} and \cite[2.2]{STWW}), it seems that  the literature still lacks a systematic study of their fundamental properties; namely, topologicity, exponentiability, and descent.

Since these properties are essential to our endeavour, we give herein an exposition on  the lax comma $2$-categories of $\Ord$, the $2$-category of ordered sets (also called preordered sets).
We prove that the forgetful functor $\Ord//X\to \Ord $ is topological if and only if $X$ is complete. Moreover, we show that, provided that $X$ has a bottom element, $\Ord//X$ is complete and cartesian closed if and only if $X$ is. We end by analysing descent in this category. Namely, when $X$ is complete, we show that, for a morphism in $\Ord//X$, being pointwise effective for descent in $\Ord$ is sufficient, while being effective for descent in $\Ord$ is necessary, to be effective for descent in $\Ord//X$.

Although further enriched and $2$-dimensional aspects of lax comma objects are
essential to our project (see, for instance, \cite{CLN} for an overall view of our ongoing work's setting), they are not relevant to the present note and, hence, will not be dealt herein.

The main intent of this paper is threefold: (1) give an exposition of lax comma $2$-categories of $\Ord$, showing some of its nice properties; (2) provide background to our future work in descent and Galois theory regarding $\Ord$-enriched categories; (3) give a guiding template for our most general systematic study of lax comma $2$-categories. Finally, we also want to pick the community's attention to the problem of studying lax comma $2$-categories, showing that, even in the case of $\Ord$,  there are still facets to be better explored.

\section{The forgetful functor $U\colon\Ord//X\to X$}\label{sect1:forgetful-functor}

Let $X$ be an ordered set. Here by \emph{order} it is meant a reflexive and transitive binary relation, not necessarily antisymmetric (also called preorder).

The category $\Ord//X$ has as objects monotone maps $a\colon Y\to X$, where $Y$ is an ordered set, and as morphisms $f\colon(Y,a)\to(Z,b)$ monotone maps $f\colon Y\to Z$ such that $a\leq bf$:
\[\xymatrix{Y\ar@{}[rrd]|{\leq}\ar[rd]_a\ar[rr]^f&&Z\ar[ld]^b\\
&X&}\]
with the usual composition. Given two morphisms $f,g\colon(Y,a)\to(Z,b)$,  we say that $f\leq g$ if
$f(y)\leq g(y)$ for all $y\in Y$, that is to say, if $f\leq g$ in $\Ord$. This makes $\Ord // X $ an $\Ord$-enriched category.

The category $\Ord/X$ is a non-full subcategory of $\Ord//X$, having the same objects, and morphisms $f\colon(Y,a)\to(Z,b)$ those morphisms in $\Ord//X$ such that $a=bf$. The $\Ord$-enrichment in $\Ord / X$ is the same as  $\Ord//X$; that is, the inclusion  $\Ord/X\to  \Ord//X $ is locally full.

\textit{These two categories have very different behaviour, as we will see throughout this text.}
We  start  by comparing the two  locally full  forgetful functors of the diagram
\[\xymatrix{\Ord/X\ar[rd]_{\overline{U}}\ar[rr]&&\Ord//X\ar[ld]^U\\
&\Ord&}\]
It is well-known that \emph{the forgetful functor $\overline{U}\colon\Ord/X\to\Ord$ is  $\Ord$-comonadic, and therefore it reflects isomorphisms, and  creates  ($\Ord$-weighted) colimits and absolute equalizers}. Moreover, the category $\Ord/X$ is complete but $\overline{U}$ does not preserve limits in general; indeed, it preserves equalizers and pullbacks but not products: in $\Ord/X$ the terminal object is $1_X\colon X\to X$, and products are formed via pullbacks.
On the contrary, \emph{the forgetful functor $U\colon\Ord//X\to\Ord$ does not reflect isomorphisms, but in turn it is topological} \cite{H}  whenever $X$ is complete, as we show in the sequel.

We recall that a functor $G\colon\A\to\B$ is \emph{topological} if every family $(f_i\colon B\to GA_i)_{i\in I}$ of $\B$-morphisms (where $I$ may be a proper class) has a $G$-initial lift $(\overline{f}_i\colon A\to A_i)_{i\in I}$, so that $G\overline{f}_i=f_i$ for every $i\in I$ and, for every family $(h_i\colon C\to A_i)_{i\in I}$ of $\A$-morphisms and every $\B$-morphism $t\colon GC\to B$ with $f_i\cdot t=Gh_i$ for every $i\in I$, there exists a unique morphism $\overline{t}\colon C\to A$ in $\A$ with $G\overline{t}=t$ and $\overline{f}_i\,\overline{t}=h_i$ for every $i\in I$.

From now on, for every $Y\in\Ord$ and $x\in X$, we will denote by $x\colon Y\to X$ the constant (monotone) map assigning $x$ to every element of $Y$.

\begin{lemma}
\begin{enumerate}
\item If $X$ has a bottom element $\bot$, then $U\colon\Ord//X\to\Ord$ is a right adjoint.
\item If $X$ has a top element $\top$, then $U\colon\Ord//X\to\Ord$ is a left adjoint.
\end{enumerate}
\end{lemma}

\begin{proof}
It is easily checked that the functor $L\colon\Ord\to\Ord//X$, defined by $L(Y)=(Y,\bot)$ and $L(f)=f$, is a left adjoint for $U$, while $R\colon\Ord\to\Ord//X$, defined by $R(Y)=(Y,\top)$ and $R(f)=f$, is a right adjoint for $U$.
\end{proof}

\begin{theorem}\label{theo:fundamental-theorem}
The forgetful functor $U\colon\Ord//X\to\Ord$ is topological if, and only if, $X$ is a complete ordered set.
\end{theorem}
\begin{proof}
Let us assume that $X$ is complete.
Given a family $(f_i\colon Y\to(Z_i,b_i))_{i\in I}$ of monotone maps, we define $a\colon Y\to X$ by $a(y)=\bigwedge_{i\in I}\,b_i(f_i(y))$. Then, by construction, $a$ is monotone and $a\leq bf$; that is, for each $i\in I$, $f_i\colon(Y,a)\to (Z_i,b_i)$ is a morphism in $\Ord//X$. Moreover, given any family of morphisms $(g_i\colon (W,c)\to(Z_i,b_i))_{i\in I}$ and a monotone map $h\colon W\to Y$ such that $f_i h=g_i$ for every $i$, then it is easily checked that $c\leq ah$, i.e. $h\colon(W,c)\to(Y,a)$ is a morphism in $\Ord//X$ (and clearly the unique whose image under $U$ is $h\colon W\to Y$).

It remains to check that $X$ is complete provided that $U$ is topological. Let $(x_i)_{i\in I}$ be a family of elements of $X$ and consider the identities $(1\to (1,x_i))_{i\in I}$. The $U$-initial structure $x$ on $1$, with respect to this family, is clearly $\bigwedge_{i\in I}\,x_i$ in $X$: so that $(1,x)\to(1,x_i)$ is a morphism, $x\leq x_i$ for every $i$; the universal property of the lifting gives that, if $y\in X$ is such that $y\leq x_i$ for every $i\in I$, then $y\leq x$.
\end{proof}

\begin{corollary}
\begin{enumerate}
\item If $X$ has a bottom element, then $X$ is complete if, and only if, $\Ord//X$ is complete.
\item If $X$ has a top element, then $X$ is complete if, and only if, $\Ord//X$ is cocomplete.
\end{enumerate}
\end{corollary}

\begin{proof}
(1) If $X$ is complete, then the forgetful functor $\Ord//X\to \Ord$ is topological, and therefore $\Ord//X$ is complete since $\Ord$ is.
To prove the converse, let $(x_i)_{i\in I}$ be a family of elements of $X$, and let $(Y,a)$ be the product of $(x_i\colon 1\to X)_{i\in I}$ in $\Ord//X$. Then $Y=U(Y,a)$ is the product of $(U1)_{i\in I}$ in $\Ord$, that is $Y$ is a singleton $\{y\}$. The universal property of the product in $\Ord//X$ means exactly that $a(y)$ is the infimum in $X$ of the family $(x_i)_{i\in I}$.

 The proof of (2) is analogous.
\end{proof}


\emph{From now on, $X$ is a complete ordered set.}

 We find it worth to describe how limits and colimits are built in $\Ord//X$. Given a family $(X_i,a_i)_{i\in I}$ of objects of $\Ord//X$, the structure $a\colon\prod X_i\to X$ in the product $\prod X_i$ is defined by $a\left( (x_i)_i\right) =\bigwedge_ia_i(x_i)$, while the structure in its coproduct $\coprod X_i$ is given by $b\colon \coprod X_i\to X$, with $b(y)=a_i(y)$ when $y\in X_i$. Equalisers are built as expected: given morphisms $f,g\colon (Y,a)\to (Z,b)$, its equaliser is $m\colon(M=\{y\in Y\,;\,f(y)=g(y)\},\widehat{a})\to(Y,a)$, where $\widehat{a}$ is the restriction of $a$ to $M$. Coequalisers are given by Kan extensions, as we show next. For this result we only need that $X$ has a top element.

\begin{lemma}
Given morphisms $f,g,h$ in $\Ord//X$ as in the diagram
\[\xymatrix{Y\ar[dr]_a\ar@<2pt>[r]^f\ar@<-2pt>[r]_g&Z\ar@{}[rd]|(0.35){\leq}\ar[r]^h\ar[d]^b&W\ar[dl]^c\\
\ar@{}[ru]|(0.65){\leq}&X&}\]
h is the coequaliser of $f,g$ in $\Ord//X$ if, and only if:
\begin{enumerate}
\item[(1)] $h$ is the coequaliser of $f,g$ in $\Ord$;
\item[(2)] $c$ is the left Kan extension of $b$ along $h$.
\end{enumerate}
\end{lemma}

\begin{proof}
 Assume that $h\colon(Z,b)\to(W,c)$ is the coequaliser of $f,g$ in $\Ord//X$. Since the forgetful functor into $\Ord$ is \emph{a left adjoint}, $h$ is the coequaliser of $f,g$ in $\Ord$. To check that $c$ is the right Kan extension of $b$ along $h$, let $c'\colon W\to X$ be such that $b\leq c'  h$. Then, by the universal property of the coequaliser, there exists $t\colon(W,c)\to(W,c')$ such that $t  h=h$; surely $t=1_W$ since $h$ is an epimorphism, and therefore $c\leq c'$ as claimed.

 Conversely, assume that $h\colon (Z,b)\to(W,c)$ satisfies conditions (1), (2). Then trivially $h  f=h  g$, and, for any $h'\colon (Z,b)\to(W',c')$ with $h'  f=h'  g$ there is a (unique) monotone map $t\colon W\to W'$ such that $t  h=h'$. Since, by assumption, $b\leq c'  h'=c'  t  h$, by (2) we conclude that $c\leq c'  t$, that is, $t\colon (W,c)\to (W',c')$ is a morphism in $\Ord//X$ and our conclusion follows.
\end{proof}
\begin{remark}[Weighted (co)limits]
	We refer to \cite[pages~7 \& 8]{ASV} for $\Ord $-enriched weighted (co)limits.
	Recall that an $\Ord $-enriched category is $\Ord$-(co)complete whenever it has conical (co)limits and the so-called $\Ord$-(co)tensors, called herein $\Ord$-(co)powers (see, for instance, \cite[Theorem~3.73]{Kel05} for the general enriched setting).

	We establish herein that, when $X$ is (co)complete,
	$ \Ord //  X $ is $\Ord$-complete and cocomplete.
	More precisely, assuming that $X$ is (co)complete, the conical (co)limits described in Section \ref{sect1:forgetful-functor}  are $\Ord$-enriched. Furthermore, for a pair $\left( W, \left( Y, a\right) \right)\in \Ord \times \Ord // X $:
	\begin{enumerate}
		\item[--] the $\Ord$-copower $W\otimes \left( Y, a\right)$ is given by $  \left( W\times Y, W\otimes a\right) $ where $W\otimes a (w,y) = a(y) $.
		\item[--] the $\Ord $-power $W\pitchfork \left( Y, a\right)$ is given by $\left( Y ^ W, a^\top\right)$ where $$\displaystyle a^\top (f) =\bigwedge_{w\in W}\,a(f(w)), $$
		in which $Y^W$ is the exponential in $\Ord $.
	\end{enumerate} 	
\end{remark}

\section{Exponentiability}
In order to investigate under which conditions $\Ord//X$ is a cartesian closed category, we first recall two well-known results.

First of all, the (complete) ordered set $X$, as a thin category, is cartesian closed if, and only if, it has a binary operation $X\times X\to X$, assigning to each pair $(x,y)$ an element $y^x$ such that $z\leq y^x$ if and only if $z\wedge x\leq y$, for every $x,y,z\in Z$.  This is in fact equivalent, for antisymmetric orders, to $X$ being a frame, i.e. in $X$ arbitrary joins distribute over finite meets, or a complete Heyting algebra.

Secondly, $\Ord$ is a cartesian closed category. For each ordered set $Y$, the right adjoint functor $(\;)^Y$ to $(\;)\times Y\colon\Ord\to\Ord$ assigns to each ordered set $Z$ the set \[Z^Y=\{f\colon Y\to Z\,;\,f\mbox{ is a monotone map}\},\] equipped with the pointwise order; that is, for $f,g\in Z^Y$, $f\leq g$ if, for all $y\in Y$, $f(y)\leq g(y)$.

On the contrary, $\Ord/X$ is not cartesian closed in general. The following result can be found in \cite{Th}.

\begin{theorem}
Given a monotone map $a\colon Y\to X$, the functor $(\;)\times(Y,a)\colon\Ord/X\to\Ord/X$  has a right adjoint if, and only if,
\[(\forall y_0\leq y_1\mbox{ in }Y)\;(\forall x\in X)\;:\;a(y_0)\leq x\leq a(y_1)\;\Rightarrow\;(\exists y\in Y)\;:\;y_0\leq y\leq y_1\mbox{ and }a(y)=x.\]
\end{theorem}

Again, $\Ord//X$ behaves differently:
\begin{theorem}
The following assertions are equivalent:
\begin{tfae}
\item $X$ is cartesian closed;
\item $\Ord//X$ is cartesian closed.
\end{tfae}
\end{theorem}
\begin{proof}
(i)$\Rightarrow$(ii): Given two objects $a\colon Y\to X$ and $b\colon Z\to X$ in $\Ord//X$, in order to define $(Z,b)^{(Y,a)}$ first we consider the ordered set $Z^Y$ as defined above, and then the map $b^a\colon Z^Y\to X$ defined by
\[b^a(f)=\bigwedge_{y\in Y}\,b(f(y))^{a(y)}.\]
The map $b^a$ is monotone: if $f,g\colon Y\to Z$ are monotone maps, with $f\leq g$, then, for every $y\in Y$,
\[b(f(y))^{a(y)}\wedge a(y)\leq b(f(y))\leq b(g(y))\;\Rightarrow\;b(f(y))^{a(y)} \leq b(g(y))^{a(y)}.\]
The monotone map $\ev$ is a morphism in $\Ord//X$
\[\xymatrix{Z^Y\times Y\ar@{}[rrd]|{\leq}\ar[rd]_{b^a\wedge a}\ar[rr]^-\ev&&Z\ar[ld]^b\\
&X&}\]
since, by definition of $b^a$, for all $f\in Z^Y$ and $y\in Y$,
\[b^a(f)\wedge a(y)=\bigwedge_{y'\in Y}\,b(f(y'))^{a(y')}\wedge a(y)\leq b(f(y))^{a(y)}\wedge a(y)\leq b(f(y)).\]
To check its universality let $c\colon W\to X$ be an object and $h\colon (W,c)\times(Y,a)\to(Z,b)$ a morphism in $\Ord//X$. Then $\overline{h}\colon W\to Z^Y$, with $\overline{h}(w)\colon Y\to Z$ defined by $\overline{h}(w)(y)=h(w,y)$ for every $w\in W$ and $y\in Y$, is a morphism in $\Ord//X$:
\[c(w)\wedge a(y)\leq b(h(w,y))\Rightarrow c(w)\leq b(\overline{h}(w)(y))^{a(y)},\]
hence \[c(w)\leq \bigwedge_{y\in Y}\,b(\overline{h}(w)(y))^{a(y)}.\]
Therefore $(\;)\times (Y,a)$ has a right adjoint $(\;)^{(Y,a)}$ assigning $(Z^Y,b^a)$ to each $(Z,b)$ in $\Ord//X$.

(ii)$\Rightarrow$(i): Assuming that $\Ord//X$ is cartesian closed,  for each $x\in X$, let $(W,c)=(X,1_X)^{(1,x)}$ be the exponential in $\Ord//X$. Then the ordered set $W$ is isomorphic to $X$, since:
{\small \[W\cong\Ord//X((1,\bot),(X,1_X)^{(1,x)})\cong\Ord//X((1,\bot)\times(1,x),(X,1_X))\cong \Ord//X((1,\bot),(X,1_X))\cong X.\]}
We will show that $y^x=c(y)$, where $y\colon 1\to X$ is the map assigning $y$ to the only element of $1$,  for $y\in X$. Using
\[\xymatrix{W\times 1\ar@{}[rrd]|{\leq}\ar[rd]_{c\wedge x}\ar[rr]^-\ev&&X\ar[ld]^{1_X}\\
&X&}\]
one concludes that $c(y)\wedge x\leq y$; moreover, from the universality of $\ev$ it follows that, if $z\wedge x\leq y$, then the morphism $y\colon(1,z)\times(1,x)\to(X,1_X)$.
\[\xymatrix{1\times 1\ar@{}[rrd]|{\leq}\ar[rd]_{z \wedge x}\ar[rr]^-y&&X\ar[ld]^{1_X}\\
&X&}\] induces, by universality of $\ev$, a morphism  $\overline{u}\colon (1,z)\to (W,c)$ such that $\ev  (\overline{u}\times 1)=y$, and thus $z\leq c(y)$ as required.
\end{proof}

A careful analysis of this proof allows us to conclude the following

\begin{corollary}\label{coro:exponential}
Let $(Y,a)$ an object of $\Ord//X$. The following conditions are equivalent:
\begin{tfae}
\item $(Y,a)$ is exponentiable in $\Ord//X$;
\item For all $y\in Y$, $a(y)$ is exponentiable in $X$.
\end{tfae}
\end{corollary}

\begin{proof}
(i) $\Rightarrow$ (ii) is shown exactly as in the Theorem above, observing that to define the exponentials with exponent $(Y,a)$ we only need exponentials in $X$ with exponent $a(y)$, for $y\in Y$.

(ii) $\Rightarrow$ (i): Assuming that $(Y,a)$ is exponentiable, let $(W,c)=(X,1_X)^{(Y,a)}$  and let $\ev\colon (W,c)\times(Y,a)\to (X,1_X)$ be the counit of the adjunction. Then, as before, it is easy to check that $W=\{f\colon Y\to X\,;\, f$ is monotone$\}$:
{\small \[W\cong\Ord//X((1,\bot),(W,c))\cong \Ord//X((1,\bot)\times(Y,a),(X,1_X))\cong\Ord//X((Y,\bot),(X,1_X))\cong\Ord(Y,X).\]} Moreover, the isomorphism
\[\Psi\colon\Ord//X((1,\bot),(W,c))\to\Ord//X((1,\bot)\times(Y,a),(X,1_X))\]
is given by \[\Psi(g\colon(1,\bot)\to(W,c))\colon (1,\bot)\times(Y,a)\to(X,1_X),\;(*,y)\mapsto \ev(g,y),\]
for every $g\colon(1,\bot)\to(W,c)$. That is, the counit can be defined as $\ev(g,y)=g(y)$. Let $x\in X$, and  $g\colon Y\to X$ be defined by $g(y)=x$ for all $y\in Y$. Then, on one hand, $c(g)\wedge a(y)\leq x$ because $\ev\colon (W,c)\times (Y,a)\to (X,1_X)$ is a morphism in $\Ord//X$, and, on the other hand, if $z\in X$ is such that $z\wedge a(y)\leq x$ then the map $h\colon (1,z)\times(Y,a)\to (X,1_X)$ constantly equal to $x$ is a morphism in $\Ord//X$ and so there is $\overline{h}\colon(1,z)\to(W,c)$ such that $\ev(\overline{h}\times 1_Y)=h$. Necessarily $\overline{h}(*)=g$ and therefore $z\leq c(g)$.
\end{proof}

\section{Descent}

A morphism $p\colon E\to B$ in a category $\Cc$ with pullbacks is said to be of \emph{effective descent} if the change-of-base functor $p^*\colon \Cc/B\to \Cc/E$ is monadic. The study of effective descent morphisms in (topological) categories has a long history, with several notable contributions such as \cite{JT91, RT, JS, CH02, JST}.

The characterization of effective descent morphisms in $\Ord$ was established in \cite{JS}. This characterization naturally extends to the comma categories, where a morphism $f$ in the comma category $\Ord/X$ is of effective descent if and only if its underlying morphism in $\Ord$ is of effective descent. This is due to the natural isomorphism between $\Ord/X/(Y,a)$ and $\Ord/Y$. However, this simplification does not apply to lax comma ordered categories.

Therefore, in this section, we will focus on the study of effective descent morphisms in $\Ord//X$.
In particular we will show that being effective for descent in $\Ord$ is necessary but not sufficient for a morphism to be effective for descent in $\Ord//X$.


We start by characterizing (pullback stable) regular epimorphisms in $\Ord//X$.

\begin{lemma}
For a morphism $f\colon(Y,a)\to(Z,b)$ in $\Ord//X$, the following conditions are equivalent:
\begin{tfae}
\item $f$ is a regular epimorphism in $\Ord//X$;
\item $f$ is a regular epimorphism in $\Ord$ and
\begin{equation}\label{eq:regepi}
(\forall z\in Z)\;\;b(z)=\bigvee_{f(y)\leq z}\,a(y).
\end{equation}
\end{tfae}
\end{lemma}

\begin{proof}
What remains to show is that \eqref{eq:regepi} is equivalent to $b=\lan_f a$, that is, $b$ is the left Kan extension of $a$ along $f$. Given a regular epimorphism $f\colon Y\to Z$ in $\Ord$ and a monotone map $a\colon Y\to X$, \eqref{eq:regepi} defines a monotone map $b\colon Z\to X$ such that $a\leq b f$. Moreover, if $a\leq b'f$ for some monotone map $b'\colon Z\to X$, then, for every $z\in Z$ and $y\in Y$ with $f(y)\leq z$,
$a(y)\leq c'(f(y))\leq c'(z)$,
and so $c\leq c'$. The converse is shown analogously.
\end{proof}

Recall that a morphism $g: Y\to Z $ is a (pullback) stable regular epimorphism in $\Ord$ if, and only if, for each $z_0\leq z_1$ in $Z$ there exist $y_0\leq y_1$ in $Y$ such that $g(y_i)=z_i$ ($i=0,1$).

\begin{proposition}
For a morphism $f\colon(Y,a)\to(Z,b)$ in $\Ord//X$, consider the following conditions:
\begin{tfae}
\item $f$ is a stable regular epimorphism in $\Ord//X$;
\item $Uf$ is a stable regular epimorphism in $\Ord$ and
\begin{equation}\label{eq:stableregepi}
(\forall z\in Z)\;\;b(z)=\bigvee_{f(y)= z}\,a(y).
\end{equation}
\end{tfae}
Then \emph{(i)$\Rightarrow$(ii)}, and \emph{(i)$\Leftrightarrow$(ii)} provided that $X$ is cartesian closed.
\end{proposition}

\begin{proof}
(i) $\Rightarrow$ (ii): The forgetful functor $U\colon\Ord//X\to \Ord$ preserves regular epimorphisms and pullbacks, hence every stable regular epimorphism in $\Ord//X$ is also stably a regular epimorphism in $\Ord$. If \eqref{eq:stableregepi} does not hold, that is, if there exists $z\in Z$ with $b(z)> \bigvee_{f(y)=z}\,a(y)$, then we consider the pullback of $f$ along $g\colon(1,b(z))\to(Z,b)$ with $g(*)=z$. It is easy to check that in the pullback diagram
\[\xymatrix{f^{-1}(z)\ar@/_4pc/[rdd]_{a\pi_1}\ar[rr]^-{\pi_2}\ar[d]^{\pi_1}&&1\ar[d]^g\ar@/^4pc/[ldd]^{b(z)}\\
Y\ar@{}[rrd]|{\leq}\ar[rd]_a\ar[rr]^f&&Z\ar[ld]^b\\
&X&}\]
$\pi_2$ is not a regular epimorphism in $\Ord//X$ since it does not satisfy \eqref{eq:regepi}.\\

(ii) $\Rightarrow$ (i): If $f\colon(Y,a)\to(Z,b)$ satisfies (ii), for any pullback diagram in $\Ord//X$
\[\xymatrix{Y\times_ZW\ar@/_4pc/[rdd]_{a\pi_1\wedge c\pi_2}\ar[rr]^-{\pi_2}\ar[d]^{\pi_1}&&W\ar[d]^g\ar@/^4pc/[ldd]^{c}\\
Y\ar@{}[d]|\leq\ar@{}[rrd]|{\leq}\ar[rd]_a\ar[rr]^f&&Z\ar@{}[d]|\geq\ar[ld]^b\\
&X&}\]
by assumption $\pi_2$ is a regular epimorphism in $\Ord$, so it remains to be shown that $\pi_2$ satisfies \eqref{eq:regepi}: for any $w\in W$,
$c(w)\leq b(g(w))=\bigvee_{f(y)=g(w)}\,a(y)$; hence
\[\begin{array}{rcll}
c(w)\leq b(g(w))\wedge c(w)&=&(\displaystyle\bigvee_{f(y)=g(w)}\,a(y))\wedge c(w)\vspace*{2mm}\\
&=&\displaystyle\bigvee_{f(y)=g(w)}\,(a(y)\wedge c(w))&\mbox{(because $X$ is cartesian closed)}\vspace*{2mm}\\
&\leq&\displaystyle\bigvee_{f(y')=g(w'),\,w'\leq w}\,(a\pi_1\wedge c\pi_2)(y',w').
\end{array}\]
\end{proof}

\pagebreak
Next we investigate effective descent morphisms in $\Ord//X$. We will show that, for a given morphism $f\colon(Y,a)\to(Z,b)$ in $\Ord//X$,

\begin{equation}\label{eq:PED}
(\forall z_0\leq z_1\leq z_2\mbox{ in }Z)\;(\exists y_0\leq y_1\leq y_2\mbox{ in }Y)\;:\;f(y_i)=z_i\;(i=0,1,2)\mbox{ and }a(y_0)=b(z_0)
\end{equation}
\[\Downarrow\]
\[f\mbox{ is effective for descent in }\Ord//X\]
\[\Downarrow\]
\begin{equation}\label{eq:ED}
(\forall z_0\leq z_1\leq z_2\mbox{ in }Z)\;(\exists y_0\leq y_1\leq y_2\mbox{ in }Y)\;:\;f(y_i)=z_i\;(i=0,1,2).
\end{equation}

We start by showing the latter implication.

\begin{theorem}
If $f\colon(Y,a)\to(Z,b)$ is effective for descent in $\Ord//X$, then $Uf\colon Y\to Z$ is effective for descent in $\Ord$.
\end{theorem}
\begin{proof}
With $f\colon(Y,a)\to(Z,b)$ also its pullback $f_\bot$ along $(Z,\bot)\to(Z,b)$
\[\xymatrix{(Y,\bot)\ar[r]^{f_\bot}\ar[d]_1&(Z,\bot)\ar[d]^1\\
(Y,a)\ar[r]_f&(Z,b)}\]
is effective for descent. Observing that the change-of-base functors of $f_\bot$ in $\Ord//X$ and of $Uf_\bot=Uf$ in $\Ord$ are isomorphic:
\[(\xymatrix{(\Ord//X)/(Z,\bot)\ar[rr]^{f_\bot^*}&&(\Ord//X)/(Y,\bot)})\cong(\xymatrix{\Ord/Z\ar[rr]^{(Uf)^*}&&\Ord/Y})\]
we conclude that $Uf$ is effective for descent in $\Ord$, which, thanks to \cite[Proposition 3.4]{JS}, is equivalent to \eqref{eq:ED}.
\end{proof}

To show that \eqref{eq:PED} is sufficient for $f$ to be effective for descent, we will make use of the chain of pullback preserving (faithful) inclusions
\begin{equation}\label{eq:chain}
\xymatrix{\Ord//X\ar[r]^-\Pi&[X^\op,\Ord]\ar[r]&[X^\op,\Rel]\ar[r]&[X^\op,\Gph],}
\end{equation}
where
\begin{itemize}
\item $\Pi(Y,a)\colon X^\op\to\Ord$ is defined by $\Pi(Y,a)(x)=Y_x=\{y\in Y\,;\,x\leq a(y)\}$, $\Pi(Y,a)(x\geq x')$ is the inclusion of $Y_x$ in $Y_{x'}$ and $\Pi(f\colon(Y,a)\to(Z,b))_x$ is the (co)restriction $Y_x\to Z_x$ of $f$,
\item  $\Rel$ is the category having as objects pairs $(X,R_X)$, where $X$ is a set and $R_X\subseteq X\times X$ is a binary relation on $X$, and as morphisms maps preserving the binary relation, while
    \item $\Gph$ is the category of parallel pairs of morphisms $\xymatrix{R\ar@<2pt>[r]^u\ar@<-2pt>[r]_v&X}$, having as morphisms $(f,g)\colon (u,v)\to (u',v')$ pairs of maps $f\colon R\to R'$, $g\colon X\to X'$ such that $u' f=g u$ and $v' f=g v$.
 \end{itemize}
 so that we have a chain of (full) embeddings $\Ord\to\Rel\to\Gph$.

The following theorem, that can be found, for instance, in \cite[pag.~260]{JT} or \cite[Theorem~1.4]{FLND}, is also essential for the proof of our claim.

\begin{theorem}\label{th:obstruction}
Let $\A$ and $\B$ be categories with pullbacks. If $F\colon \A\to\B$ is a
fully faithful pullback preserving functor and $F(f)$ is of effective descent in $\B$, then $f$ is
of effective descent if, and only if, it satisfies the following property: whenever the diagram
below is a pullback in $\B$, there is an object $A$ in $\A$ such that $F(A)\cong B$
\[\xymatrix{F(P)\ar[r]\ar[d]&B\ar[d]\\
F(Y)\ar[r]_{F(f)}&F(Z).}\]
\end{theorem}

Indeed, using \eqref{eq:chain} we will show that in $[X^\op,\Ord]$ a natural transformation is effective for descent if, and only if, it is pointwise effective for descent in $\Ord$, and that $f$ is effective for descent in $\Ord//X$ provided that $\Pi f$ is effective for descent in $[X^\op,\Ord]$.

\begin{proposition}
In $[X^\op,\Gph]$ a morphism $\alpha\colon F\to G$ is effective for descent if, and only if, it is an epimorphism.
\end{proposition}
\begin{proof}
The category $[X^\op,\Gph]$ is a topos.
\end{proof}

\begin{proposition}
For a morphism $\alpha\colon F\to G$ in $[X^\op,\Rel]$, the following conditions are equivalent:
\begin{tfae}
\item $\alpha$ is effective for descent;
\item $\alpha$ is a stable regular epimorphism;
\item $\alpha$ is a regular epimorphism;
\item $(\forall x\in X)\;\alpha_x$ is a regular epimorphism in $\Rel$;
\item $(\forall x\in X)\;(\forall (z_0,z_1)\in G(x))\;(\exists (y_0,y_1)\in F(x))\;:\;\alpha_x(y_0,y_1)=(z_0,z_1)$;
\item $(\forall x\in X)\;\alpha_x$ is effective for descent in $\Rel$.
\end{tfae}
\end{proposition}
\begin{proof}
Applying Theorem \ref{th:obstruction} for the inclusion $[X^\op,\Rel]\to [X^\op,\Gph]$, and knowing that pullbacks in $[X^\op,\Rel]$ are formed pointwise and regular epimorphisms are pullback stable, one concludes that (i)$\Leftrightarrow$(ii)$\Leftrightarrow$(iii)$\Leftrightarrow$(iv). The characterizations of regular epimorphisms and effective descent morphisms in $\Rel$ of \cite[Propositions 2.1 and 3.3]{JS} give (iv)$\Leftrightarrow$(v)$\Leftrightarrow$(vi).
\end{proof}

\begin{theorem}
In $[X^\op,\Ord]$ a morphism $\alpha\colon F\to G$ is effective for descent if, and only if,
\begin{equation}\label{eq:ped}
(\forall x\in X)\;\alpha_x\mbox{ is effective for descent in $\Ord$.}
\end{equation}
\end{theorem}
\begin{proof}
Now we apply Theorem \ref{th:obstruction} to the full inclusion $[X^\op,\Ord]\to [X^\op,\Rel]$. Since it preserves pullbacks, to prove that $\alpha\colon F\to G$ satisfying \eqref{eq:ped} is effective for descent in $[X^\op,\Ord]$ it is sufficient to show that in the pullback diagram
\[\xymatrix{F\times_GH\ar[r]^-{\rho}\ar[d]_{\pi}&H\ar[d]^\beta\\
F\ar[r]_\alpha&G}\]
if $F\times_GH$ belongs to $[X^\op,\Ord]$, then also $H$ does. For each $x\in X$, consider the pullback diagram
\[\xymatrix{F(x)\times_{G(x)}H(x)\ar[r]^-{\rho_x}\ar[d]_{\pi_x}&H(x)\ar[d]^{\beta(x)}\\
F(x)\ar[r]_{\alpha_x}&G(x).}\]
If $\alpha_x$ is effective for descent in $\Ord$, then $H(x)\in\Ord$ since $F(x)\times_{G(x)}H(x)$ does by assumption.

Conversely, let us assume that $\alpha$ is effective for descent, and let $x\in X$ and $z_0\leq z_1\leq z_2$ in $G(x)$. Consider the functor $H\colon X^\op\to\Rel$ defined by
\[H(x')=\left\{\begin{array}{ll}
(\{z_0,z_1,z_2\},\{(z_0,z_0), (z_1,z_1), (z_2,z_2), (z_0,z_1),(z_1,z_2)\})&\mbox{ if }x'\cong x\\
(\{z_0\},\{(z_0,z_0)\})&\mbox{ if }x'<x\\
\emptyset&\mbox{ otherwise},
\end{array}\right.\]
with $H(x''\geq x')\colon H(x'')\to H(x')$ given by $\emptyset\to H(x')$ if $x''\not\leq x$, the constant map $H(x'')\to H(x')$ if $x'<x$ and $x''\leq x$, and the identity otherwise. Since by assumption $\alpha$ is effective for descent and $H$ does not belong to $[X^\op,\Ord]$ ($H(x)$ is not transitive), also $F\times_GH$ does not belong to $[X^\op,\Ord]$. If $x'\not\cong x$, then $F(x')\times_{G(x')}H(x')$ is either $\emptyset$ or isomorphic to $F(x')$, hence an ordered set. Therefore there must exist $x'\cong x$ (and so we may consider $x'=x$ since images of isomorphic elements will be isomorphic too) so that  the binary reflexive relation in $F(x)\times_{G(x)}H(x)$, that is,
 \[\{((y,y),(z_i,z_i));\, \alpha_x(y)=z_i,\,i=0,1,2\}\,\cup\,\{((y,y'),(z_i,z_{i+1}));\,\alpha_x(y,y')=(z_i,z_{i+1}),\,i=0,1\}
\]
is not an order. Failure of transitivity at $F(x)\times_{G(x)}H(x)$ means that, necessarily, there exist $((y,y'),(z_0,z_1))$ and $((y',y''),(z_1,z_2))$ in $F(x)\times_{G(x)}H(x)$; then $\pi_x(y,y')$ and $\pi_x(y',y'')$ gives that $y\leq y'\leq y''$ in $F(x)$. Computing now $\alpha_x$ gives $\alpha_x(y)=z_0$, $\alpha_x(y')=z_1$ and $\alpha_x(y'')=z_2$.
\end{proof}

\begin{lemma}
Let $\Pi\colon\Ord//X\to[X^\op, \Ord]$ the functor defined in \eqref{eq:chain}. Then $\Pi$ is a full and faithful right adjoint.
\end{lemma}
\begin{proof}
For each functor $H\colon X^\op\to\Ord$ and $x\in X$, let $H[x]=H(x\geq\bot)(H(x))\subseteq H(\bot)$.

The left adjoint $L\colon [X^\op,\Ord]\to\Ord//X$ of $\Pi$ is defined by $L(H)=(H(\bot),d_H)$ where $d_H\colon H(\bot)\to X$ is defined by $d_H(w)=\bigvee\{x\in X\,;\,w\in\,\uparrow H[x]\}$, and $L(\alpha\colon H\to H')=\alpha_\bot$. Indeed, $d_H$ is clearly a monotone map; moreover, if $\alpha\colon H\to H'$ is a natural transformation, then $\alpha_\bot\colon H(\bot)\to H'(\bot)$ is monotone; to show that
\[\xymatrix{H(\bot)\ar@{}[rrd]|{\leq}\ar[rd]_{d_H}\ar[rr]^{\alpha_\bot}&&H'(\bot)\ar[ld]^{d_{H'}}\\
&X&}\]
let $w\leq w'\in H[x]$; then, since the following diagram commutes
\[\xymatrix{H(x)\ar[r]^{\alpha_x}\ar[d]_{H(x\geq\bot)}&H'(x)\ar[d]^{H'(x\geq\bot)}\\
H(\bot)\ar[r]_{\alpha_\bot}&H'(\bot)}\]
$\alpha_\bot(w)\leq \alpha_\bot(w')\in H'[x]$, and so $d_H(w)\leq d_{H'}(\alpha_\bot(w))$.

To show that $L\dashv \Pi$ we define the counit of the adjunction: given $(Y,a)\in\Ord//X$, $L\Pi(Y,a)=(Y,a)$ because $L(\bot)=Y_\bot=Y$ and $d_{L\Pi(Y,a)}=\bigvee\{x\in X\,;\,y\geq y'\in Y_x\}=a(y)$, since $y\in Y_{a(y)}$. Let us show that the identity $L\Pi(Y,a)\to(Y,a)$ has the required universal property: if $H\colon X^\op\to\Ord$ is a functor and $f\colon(H(\bot),d_H)\to (Y,a)$ is a morphism in $\Ord//X$, then $\varphi\colon H\to\Pi(Y,a)$, defined for every $x\in X$ by $\varphi_x\colon H(x)\to Y_x$ with $\varphi_x(w)=f(H(x\geq\bot)(x))$, is easily checked to be the unique natural transformation such that $L\varphi=f$.

Therefore $\Pi$ is a right adjoint and, moreover, it is full and faithful since the counit is an isomorphism.
\end{proof}

\begin{theorem}
If $f\colon(Y,a)\to(Z,b)$ is a morphism in $\Ord//X$ satisfying \eqref{eq:PED}, that is,
\[(\forall z_0\leq z_1\leq z_2\mbox{ in }Z)\;(\exists y_0\leq y_1\leq y_2\mbox{ in }Y)\;:\;f(y_i)=z_i\;(i=0,1,2)\mbox{ and }a(y_0)=b(z_0),\]
then $f$ is effective for descent in $\Ord//X$.
\end{theorem}
\begin{proof}
Let $f\colon(Y,a)\to(Z,b)$ satisfy the condition above. Since $\Pi$ is full, faithful and preserves pullbacks, applying Theorem \ref{th:obstruction} what we need to show is that, given a pullback diagram in $[X^\op,\Ord]$
\[\xymatrix{\Pi(P,c)\ar[r]^\rho\ar[d]_{\Pi\pi}&G\ar[d]^\beta\\
\Pi(Y,a)\ar[r]_{\Pi f}&\Pi(Z,b)}\]
$G\cong \Pi(W,d)$ for some $d\colon W\to X$ in $\Ord//X$.

First we show that, for every $x\in X$, $G(x\geq\bot)\colon G(x)\to G(\bot)$ is an injective map:
\[\xymatrix{P_\bot\ar[rrr]^{\rho_\bot}\ar[ddd]_{\pi_\bot}&&&G(\bot)\ar[ddd]^{\beta_\bot}\\
&P_x\ar[r]^{\rho_x}\ar[lu]\ar[d]_{\pi_x}&G(x)\ar[ru]\ar[d]^{\beta_x}\\
&Y_x\ar[r]_{f_x}\ar[ld]&Z_x\ar[rd]\\
Y_\bot\ar[rrr]_{f_\bot}&&&Z_\bot.}\]
Indeed, if $w_1,w_2\in G(x)$ are such that $G(x\geq\bot)(w_1)=G(x\geq\bot)(w_2)=w$, then $\beta_x(w_1)=\beta_x(w_2)$. Let $y\in Y_x$ be such that $f_x(y)=\beta_x(w_1)$. Then $(y,w_1)$ and $(y,w_2)$ belong to $P_x$, hence they also belong to $P_\bot=P$, with $\rho_\bot(y,w_1)=\rho_\bot(y,w_2)=w$, $\pi_\bot(y,w_1)=\pi_\bot(y,w_2)=y$; hence $w_1=w_2$.
Therefore also the maps $G(x'\geq x)\colon G(x')\to G(x)$ are injective, and so we may assume they are inclusions.

Now we consider $W=G(\bot)$ and define $d\colon W\to X$ by
\[d(w)=\bigvee\{x\in X\,;\,w\in G(x)\}.\]
Then:
\begin{itemize}
\item[--] $w\in G(d(w))$: if $z=\beta_\bot(w)$, then, for all $x\in X$, if $w\in G(x)$ then $z\in Z_x$, i.e. $x\leq b(z)$; hence $d(w)\leq b(z)$, and so $Z\in Z_{d(w)}$. Let $y\in Y_{d(w)}$ be such that $f(y)=z$. Then, for all $x\in X$, if $w\in G(x)$ then $(y,w)\in P_x$, or, equivalently, $x\leq c(y,w)$, which implies $d(w)\leq c(y,w)$. Hence $w\in G(c(y,w))\subseteq G(d(w))$.
\item[--] $d$ is monotone: it follows from the fact that, for each $x\in X$, $G(x)$ is upwards-closed; indeed, if $w\leq w'$ in $W$ and $w\in G(x)$, then $\beta_\bot(w)\leq\beta_\bot(w')$ and both belong to $Z_x$. Let $y\leq y'$ in $Y_x$ be such that $f(y)=\beta_\bot(w)$ and $f(y')=\beta_\bot(w')$. Then $(y,w)\leq (y',w')$ in $P$ and $(y,w)\in P_x$ implies $(y',w')\in P_x$, since $P_x$ is upwards-closed. This gives $w'\in G(x)$ as claimed.
\end{itemize}
\end{proof}

\begin{remark}
As we pointed out at the beginning of this section, $Uf$ effective for descent in $\Ord$ does not imply $f\colon(Y,a)\to(Z,b)$ effective for descent in $\Ord//X$, since it does not even imply that $f$ is a regular epimorphism in $\Ord//X$. \emph{It is an open problem to know whether every stable regular epimorphism $f$ with $Uf$ effective for descent in $\Ord$ is effective for descent in $\Ord//X$.}
\end{remark}

\section*{Acknowledgments}
We thank the anonymous referee for their valuable feedback and contributions, which have helped improve our manuscript.

\end{document}